\documentclass[12pt]{article}
\usepackage[utf8]{inputenc}
\usepackage{amsfonts}
\usepackage{amsmath}
\usepackage[english]{babel}
\usepackage{amsthm}
\usepackage[symbol]{footmisc}
\usepackage{comment}
\usepackage{array}
\usepackage{color}
\usepackage{bm}

\newtheorem{theorem}{Theorem}[section]
\newtheorem{proposition}[theorem]{Proposition}
\newtheorem{lemma}[theorem]{Lemma}
\newtheorem{corollary}[theorem]{Corollary}
\newtheorem{example}{Example}
\newtheorem{definition}{Definition}[section]
\newtheorem{remark}{Remark}

\title{Composition operators on Gelfand-Shilov classes}
\author{Héctor Ariza, Carmen Fernández, Antonio Galbis}
\begin{document}

\maketitle

\begin{abstract} 
We study composition operators on global classes of ultradifferentiable functions of Beurling type invariant under Fourier transform. In particular, for the classical Gelfand-Shilov classes $\Sigma_d,\ d > 1,$ we prove that a necessary condition for the composition operator $f\mapsto f\circ \psi$ to be well defined is the boundedness of  $\psi'.$ We find the optimal index $d'$ for which $C_\psi(\Sigma_d({\mathbb R}))\subset \Sigma_{d'}({\mathbb R})$ holds for any non-constant polynomial $\psi.$
\end{abstract}

\section{Introduction}

The aim of this paper is to start the investigation of composition operators on global classes of ultradifferentiable functions invariant under Fourier transform. To define this classes, we can use two approaches: either use weight functions $\omega$ to control the decay properties of the function and its Fourier transform or measure the growth behavior of a function in terms of a weight sequence $(M_p)_{p\in {\mathbb N}_0}.$ It turns out that several classes can be defined by both methods, the most relevant ones being the classical Gelfand-Shilov spaces. In recent years, Gelfand-Shilov spaces and several classes of operators acting on them have been studied (see for instance \cite{capiello, cprt, debrouwere, fgt}). We refer to \cite{bmm} and \cite{nuclear} for a characterization of those weight functions for which the global spaces of ultradifferentiable functions can be also defined using weight sequences and vice versa. We first prove that, as it happens with the Schwartz class, composition operators in this setting are never compact. From this point, the behavior of composition operators on Gelfand-Shilov spaces and on the Schwartz class  is completely different, as follows from  the results in section 3. Then, in section 4, we give some sufficient conditions on a smooth function $\psi$ to guarantee that the corresponding composition operator is well defined on a given Gelfand-Shilov class. For example, for the Gelfand-Shilov classes $\Sigma_d({\mathbb R}),$ $d>1,$ a necessary condition for the composition operator $C_\psi$ to leave the class invariant is the boundedness of  $\psi'.$ On the other hand, it is possible to find the optimal index $d'$ for which $C_\psi(\Sigma_d({\mathbb R}))\subset \Sigma_{d'}({\mathbb R})$ holds for any non-constant polynomial $\psi$ (see Corollary \ref{cor:_gevrey2} and Theorem \ref{te:2}).

\subsection{Gelfand-Shilov spaces defined by a weight function}
\begin{definition}\label{def:weight}{\rm
A continuous increasing function $\omega :[0,\infty [\longrightarrow [0,\infty [$ is called a {\it weight} if it satisfies:
\begin{itemize}
\item[$(\alpha)$] there exists $K\geq 1$ with $\omega (2t) \leq K(\omega (t)+1)$ for
all $t\geq 0$,
\item[$(\beta)$] $\displaystyle\int_{0}^{\infty}\frac{\omega (t)}{1+t^{2}}\ dt < \infty $,
\item[$(\gamma)$] $\log(1+t^{2})=o(\omega (t))$ as t tends to $\infty$,
\item[$(\delta)$] $\varphi_\omega :t\rightarrow \omega (e^{t})$ is convex.
\end{itemize}
}
\end{definition}
\par\medskip
$\omega$ is extended to ${\mathbb R}$ as $\omega(x) = \omega(|x|).$ The {\it Young conjugate} $\varphi_\omega ^{*}:[0,\infty [ \longrightarrow {\mathbb R}$  of
$\varphi_\omega$ is defined by
$$
\varphi_\omega^{*}(s):=\sup\{ st-\varphi_\omega (t):\ t\geq 0\},\ s\geq 0.
$$ Then $\varphi_\omega^\ast$ is convex, $\varphi_\omega^\ast(s)/s$ is increasing and $\displaystyle\lim_{s\to \infty}\frac{\varphi_\omega^\ast(s)}{s} = +\infty.$ Moreover, for every $A>0, \, \lambda >0$ there is $C>0$ such that $$A^jj!\leq C e^{\lambda \varphi_\omega^\ast(\frac{j}{\lambda})}$$ for each $j \in {\mathbb N}_0$ (see \cite[LemmaA.1(viii)]{paley}). The weight $\omega$ is said to be a {\it strong weight} if
\begin{itemize}
\item[$(\varepsilon)$]
there exists a constant $C \geq 1$ such that for all $y > 0$ the following inequality holds
\begin{equation*}
\int_1^\infty \frac{\omega(yt)}{t^2}\ dt \leq C\omega(y) + C.
\end{equation*}
\end{itemize}

\begin{definition}{\rm
Let $\omega$ be a weight function. The Gelfand-Shilov space of Beurling type $\mathcal{S}_{(\omega)}(\mathbb{R})$ consists of those functions $f\in L^1({\mathbb R})$ with the property that $f, \widehat{f}\in C^\infty({\mathbb R})$ and
$$
q_{\lambda, j}(f):= \max\big(\sup_{x\in {\mathbb R}}|f^{(j)}(x)|e^{\lambda \omega(x)},\ \sup_{\xi\in {\mathbb R}}|\widehat{f}^{(j)}(\xi)|e^{\lambda \omega(\xi)}\big) < +\infty$$ for every $\lambda > 0, j\in {\mathbb N}_0.$
}
\end{definition}

$\mathcal{S}_{(\omega)}(\mathbb{R})$ is a Fréchet space with different equivalent systems of seminorms (see for instance \cite{seminormas}). In particular we shall use the families of seminorms
$$
p_\lambda(f):= \sup_{j,k\in {\mathbb N}_0}\sup_{x\in {\mathbb R}}|x^kf^{(j)}(x)|e^{-\lambda\varphi^\ast(\frac{j+k}{\lambda})},\ \ \lambda > 0$$ or
$$
\pi_{\lambda, \mu}(f):= \sup_{j\in {\mathbb N}_0}\sup_{x\in {\mathbb R}}|f^{(j)}(x)|e^{-\lambda\varphi^\ast(\frac{j}{\lambda}) + \mu\omega(x)},\ \ \lambda > 0, \mu > 0.$$
\par\medskip
Condition $(\beta)$ is equivalent to the existence of non-trivial functions with compact support on $\mathcal{S}_{(\omega)}(\mathbb{R})$, hence by \cite{mt}, condition $(\varepsilon)$ is equivalent to the surjectivity of the Borel map
$$
B:{\mathcal S}_{(\omega)}({\mathbb R})\to {\mathcal E}_{(\omega)}(\{0\}), f\mapsto \left(f^{(j)}(0)\right)_{j\in {\mathbb N}_0},$$ where
$$
{\mathcal E}_{(\omega)}(\{0\}) = \left\{(x_j)_j\in {\mathbb C}^{{\mathbb N}_0}:  \ \sup_j|x_j|\exp(-k\varphi_\omega^\ast(\frac{j}{k})) < \infty\ \forall k > 0\right\}.
$$
\subsection{Gelfand-Shilov spaces defined by a weight sequence}
\begin{definition}{\rm A sequence $(M_p)_{p\in {\mathbb N}_0}$ is a weight sequence if it satisfies
		\begin{itemize}
		\item[(M0)] There exists $c > 0$ such that $\big(c(p+1)\big)^p \leq M_p,\ p\in {\mathbb N}_0.$
		\item[(M1)] $M_p^2\leq M_{p-1}M_{p+1},\ p\in {\mathbb N}$ and $M_0 = 1.$
		\item[(M2)] There are $A, H > 0$ such that $M_p\leq AH^p\min_{0\leq q\leq p}M_{q}M_{p-q},\ p\in {\mathbb N}_0.$
		\item[$(\gamma_1)$] $\displaystyle\sup_p \frac{m_p}{p}\sum_{j\geq p}\frac{1}{m_j} < \infty,$ where $\displaystyle m_p = \frac{M_p}{M_{p-1}}.$
		\end{itemize}

}
\end{definition}

\begin{definition}{\rm The space ${\mathcal S}_{(M_p)}({\mathbb R})$ associated to the weight sequence $(M_p)_{p\in {\mathbb N}_0}$ consists of those functions $f\in C^\infty({\mathbb R})$ such that, for each $h > 0:$
		$$
		\sup_{x\in {\mathbb R}}\sup_{j,\ell\in {\mathbb N}_0}\frac{|x^\ell f^{(j)}(x)|}{h^{j+\ell}M_{j+\ell}} < +\infty.
		$$
}
\end{definition}
\par\medskip
Condition $(\gamma_1)$ is required since it is equivalent to the surjectivity of the Borel map
$$
B:{\mathcal S}_{(M_p)}({\mathbb R})\to \Lambda_{(M_p)}, f\mapsto \left(f^{(p)}(0)\right)_{p\in {\mathbb N}_0},$$ where
$$
\Lambda_{(M_p)} = \left\{(x_p)_p\in {\mathbb C}^{{\mathbb N}_0}:  \ \sup_p\frac{|x_p|}{h^pM_p} < \infty\ \forall h > 0\right\}.$$
\par\medskip
Condition $(\gamma_1)$ implies (M3)': $\sum_p\frac{M_{p-1}}{M_p} < \infty,$ which is equivalent to the fact that ${\mathcal S}_{(M_p)}({\mathbb R})$ contains non trivial compactly supported functions. Moreover, from \cite[Proposition 1.1]{petzsche}, $(\gamma_1)$ is equivalent to
$$
\exists Q\in {\mathbb N}:\ \lim\inf_{j\to \infty}\frac{m_{Qj}}{m_j} > 1.$$ By (\cite[Theorem 14]{bmm}), $M(t):=\sup_p\log\frac{|t|^p}{M_p},\ t\neq 0, M(0) = 1,$ is a weight function and the following assertions hold:
\par\medskip
\centerline{$\mbox{For each}\ 0<h<1\ \mbox{there exist}\ k\in {\mathbb N}\ \mbox{and}\ C > 0 \ \mbox{such that}$}
\par
\centerline{$\exp\left(k\varphi_M^\ast(p/k)\right)\leq C h^p M_p,\ p\in {\mathbb N}_0.$}
\par\medskip
\centerline{$\mbox{For each}\ k\in {\mathbb N}\ \mbox{there exist}\ 0<h<1\  \mbox{and}\ D > 0 \ \mbox{such that}$}
\par
\centerline{$h^p M_p\leq D\exp\left(k\varphi_M^\ast(p/k)\right),\ p\in {\mathbb N}_0.$}
\par\medskip
Consequently $${\mathcal S}_{(M_p)}({\mathbb R}) = {\mathcal S}_{(\omega)}({\mathbb R})$$ for $\omega = M.$ From the surjectivity of the Borel map we conclude that $\omega = M$ is a strong weight. Finally we observe that it follows from \cite[Proposition 3.6]{komatsu} that there is $H\geq 1$ such that
\begin{equation}\label{eq:fromweighttosequence}
2\omega(t)\leq \omega(Ht) + H,\ t\geq 0.
\end{equation}
\par\medskip
Conversely, \cite[Corollary 16]{bmm} and \cite[Section 3]{nuclear} imply that for any weight function $\omega$ satisfying condition (\ref{eq:fromweighttosequence}) there exists a weight sequence $(M_p)_{p\in {\mathbb N}_0}$ such that ${\mathcal S}_{(\omega)}({\mathbb R}) = {\mathcal S}_{(M_p)}({\mathbb R}).$

\begin{example}{\rm  \begin{itemize}
\item[(a)]Let $d > 1$ be given. The Gelfand-Shilov space $\Sigma_d({\mathbb R})$ is
	$$
	\Sigma_d({\mathbb R}) = {\mathcal S}_{(M_p)}({\mathbb R}) = \mathcal{S}_{(\omega)}(\mathbb{R}),$$ where
	$$
	M_p = p!^d,\ \ \omega(t) = t^{\frac{1}{d}}.$$

\item[(b)]For $s>1$ we define $\omega(t)=\rm{max}(0, \log^s t).$  These are strong weights, subadditive (or equivalent to subadditive ones) and by \cite[20.Example]{bmm} the corresponding classes ${\mathcal S}_{(\omega)}({\mathbb R})$ cannot be described by means of weight sequences.
    \end{itemize}
}
\end{example}
For completeness we recall that a weight function $\omega$ is equivalent to a subadditive weight if, and only if, $\mathcal{E}_{(\omega)}(\mathbb{R})$ is closed by composition (see \cite{fg}).

\section{Compactness}

We recall that a linear operator $T:E\to F$ between two Fréchet spaces is said to be compact if there is a $0$-neighborhood $U$ in $E$ such that $T(U)$ is a relatively compact subset of $F.$ In \cite{gj} it is proved that the composition operators in the Schwartz class are never compact. The same is true for Gelfand-Shilov classes, although the proof is very different from that of \cite{gj}.

\begin{lemma}\label{lem:faa}{\rm Let $h,\psi\in C^\infty({\mathbb R}),$ $n\in {\mathbb N}$ and $x_0\in {\mathbb R}$ such that $h^{(k)}(\psi(x_0)) = 0$ whenever $k\neq n.$ Then
		$$
		\left(h\circ \psi\right)^{(n)}(x_0) = h^{(n)}(\psi(x_0))\left(\psi'(x_0)\right)^n.$$
	}
\end{lemma}
\begin{proof}
	According to Fa\`{a} di Bruno's formula (see e.g. \cite[1.3.1]{krantz}) we have
	\begin{equation}\label{eq:faa}
	(h\circ \psi)^{(n)}(x_0) = \sum \frac{n!}{k_1!\ldots
		k_n!}h^{(k)}(\psi(x_0))\left(\frac{\psi'(x_0)}{1!}\right)^{k_1}\ldots
	\left(\frac{\psi^{(n)}(x_0)}{n!}\right)^{k_n}\end{equation} \noindent
	where the sum is extended over all $(k_1,\ldots,k_n)$ such that
	$k_1 + 2k_2 + \ldots + nk_n = n$ and we denote $k:=k_1 + \ldots + k_n.$ The only term in the summation that does not vanish is obtained when $k = n,$ therefore, $k_1 = n,\ k_2 = \ldots = k_n = 0.$ From here the conclusion follows.
\end{proof}

\begin{theorem}{\rm Let $\omega$ be a strong weight and let us assume that $\psi\in C^\infty({\mathbb R})$ satisfies $C_\psi({\mathcal S}_{(\omega)}({\mathbb R})) \subset {\mathcal S}_{(\omega)}({\mathbb R}).$ Then $C_\psi:{\mathcal S}_{(\omega)}({\mathbb R})\to{\mathcal S}_{(\omega)}({\mathbb R})$ is not a compact operator.
}
\end{theorem}
\begin{proof}
As in \cite[Lemma 2.2]{gj}, we have $\displaystyle\lim_{|x|\to \infty}\psi(x) = \infty.$ Without loss of generality we can assume there is $x_0\in {\mathbb R}$ such that $|\psi'(x_0)| > 1$ (on the contrary we consider $\sigma(x) = \psi(ax)$ for appropriate $a > 0$ and observe that the compactness of $C_\psi$ is equivalent to that of $C_\sigma$). We now assume that $C_\psi$ is compact and we fix a $0$-neighborhood $U$ in ${\mathcal S}_{(\omega)}({\mathbb R})$ with the property that $C_\psi(U)$ is relatively compact (hence bounded) in ${\mathcal S}_{(\omega)}({\mathbb R}).$ The map
$$
B_{x_0}:{\mathcal S}_{(\omega)}({\mathbb R})\to {\mathcal E}_{(\omega)}(\{0\}), f\mapsto \left(f^{(j)}(\psi(x_0))\right)_{j\in {\mathbb N}_0},$$ being the composition of a translation with the Borel map, is continuous, linear and surjective. Hence it is open and $B_{x_0}(U)$ is a $0$-neighborhood in ${\mathcal E}_{(\omega)}(\{0\}).$ Consequently there are $p\in {\mathbb N}$ and $\varepsilon > 0$ such that the conditions $(a_j)_j\in {\mathcal E}_{(\omega)}(\{0\})$ and
$$
\sup_j |a_j|\exp(-p\varphi_\omega^\ast(\frac{j}{p})) \leq \varepsilon$$ imply $(a_j)_j\in B_{x_0}(U).$ For every $n\in {\mathbb N}$ we consider $b_n\in {\mathcal E}_{(\omega)}(\{0\})$ defined by $b_n(n) = \varepsilon\exp(p\varphi_\omega^\ast(\frac{n}{p}))$ while $b_n(j) = 0$ for $j\neq n,$ and we take $f_n\in U$ such that $B_{x_0}(f_n) = b_n.$ Then
$$
\sup_{n\in {\mathbb N}}|(f_n\circ\psi)^{(n)}(x_0)|\exp(-p\varphi_\omega^\ast(\frac{n}{p})) < \infty.
	$$ From Lemma \ref{lem:faa} we conclude
	$$
	\begin{array}{*2{>{\displaystyle}l}}
\varepsilon\sup_{n\in {\mathbb N}}|\psi'(x_0)|^n & = \sup_{n\in {\mathbb N}}|f_n^{(n)}(\psi(x_0))|\cdot |\psi'(x_0)|^n\exp(-p\varphi_\omega^\ast(\frac{n}{p})) \\ & \\ & = \sup_{n\in {\mathbb N}}|(f_n\circ\psi)^{(n)}(x_0)|\exp(-p\varphi_\omega^\ast(\frac{n}{p})) < \infty \end{array}$$ which is a contradiction.
\end{proof}

\section{Negative results and necessary conditions}

\begin{lemma} {\rm
	Let $\omega$ be a weight and $(I_m)_{m\in {\mathbb N}}$ be a partition of $\mathbb{N}_0$ where $I_m$ is a finite set for all $m.$ We put  $a_j:=\exp(m\varphi_\omega^{*}(\frac{j}{m}))$ for all $j\in I_m.$ Then $\left(a_j\right)_{j\in {\mathbb N}_0}\in {\mathcal E}_{(\omega)}(\{0\}).$
}
\end{lemma}
\begin{proof}
We fix $k\in {\mathbb N}.$ Since $m\varphi_\omega^\ast\left(\frac{j}{m}\right)\leq k\varphi_\omega^\ast\left(\frac{j}{k}\right)$ for every $j\in {\mathbb N}_0$ and $m\geq k,$ then we have
$$
\sup_{j\in {\mathbb N}_0}|a_j|\exp(-k\varphi_\omega^\ast(\frac{j}{k})) = \max\{\sup_{1\leq m\leq k}\sup_{j\in I_m}\exp(m\varphi_\omega^\ast(\frac{j}{m})-k\varphi_\omega^\ast(\frac{j}{k})),1\} < +\infty.
$$
\end{proof}

For every $x\in {\mathbb R}$ we consider the translation operator
$$T_x:{\mathcal S}_{(\omega)}({\mathbb R}) \to {\mathcal S}_{(\omega)}({\mathbb R}),\ (T_xf)(t):= f(t-x).$$

\begin{lemma}\label{lem:equicontinuous}{\rm Let $(x_j)_{j\in {\mathbb N}}$ and $(\lambda_j)_{j\in {\mathbb N}}$ be two sequences of positive real numbers such that $\displaystyle\lim_j\lambda_j = +\infty.$ We consider
		$$
		U_j:{\mathcal S}_{(\omega)}({\mathbb R}) \to {\mathcal S}_{(\omega)}({\mathbb R}),\ U_jf = \exp\left(-\lambda_j\omega(x_j)\right)T_{x_j}.$$ Then $\left\{U_j:\ j\in {\mathbb N}\right\}$ is equicontinuous.
}
\end{lemma}
\begin{proof}
We recall that $\left(\pi_{n,n}\right)_{n\in {\mathbb N}}$ is a fundamental system of seminorms. From \ref{def:weight}$(\alpha)$ we get $k\in {\mathbb N}$ such that
$$
\omega(x)\leq k\left(1 + \omega(x_j) + \omega(x-x_j)\right)\ \ \forall x\in {\mathbb R}, j\in {\mathbb N}.$$
Let us fix $n\in {\mathbb N}$ and take $m_n = kn.$ Then, for every $f\in {\mathcal S}_{(\omega)}({\mathbb R}),$ we have
$$
\begin{array}{*2{>{\displaystyle}l}}
\pi_{n,n}(U_j f) & = \exp\left(-\lambda_j\omega(x_j)\right)\sup_{\ell\in {\mathbb N}_0}\sup_{x\in {\mathbb R}}|f^{(\ell)}(x-x_j)|e^{-n\varphi_\omega^\ast(\frac{\ell}{n})+n\omega(x)} \\ & \\ & \leq \exp\left(-\lambda_j\omega(x_j)\right)\pi_{m,m}(f)\sup_{x\in {\mathbb R}}\exp\left(n\omega(x)-m\omega(x-x_j)\right)\\ & \\ & \leq \exp\left(-\lambda_j\omega(x_j)\right)\pi_{m,m}(f)\exp\left(m+m\omega(x_j)\right)\\ & \\ & \leq C_n\pi_{m,m}(f)
\end{array}$$ for all $j\in{\mathbb N},$ where
$$
C_n:=\sup_{j\in {\mathbb N}}\exp\left((-\lambda_j + m)\omega(x_j) + m\right) < +\infty.$$
\end{proof}

The next result can be found in \cite{bmt}. We include a proof for completeness.

\begin{lemma}\label{lem:cambioseminorma}{\rm Let $L\in {\mathbb N}$ satisfy $\omega(et) \leq L(1+\omega(t))$ for all $t\geq 0.$ For $\lambda > 0, N\in {\mathbb N}$ we take $\mu = L^N \lambda.$ Then
		$$
		\mu\varphi_\omega^\ast(\frac{j}{\mu}) + Nj \leq \lambda\varphi_\omega^\ast(\frac{j}{\lambda}) + \lambda\sum_{k=1}^N L^k,\ j\in {\mathbb N}_0.$$

}
\end{lemma}
\begin{proof}
Take $N = 1.$ Then
$$
\begin{array}{*2{>{\displaystyle}l}}
\mu\varphi_\omega^\ast(\frac{j}{\mu}) + j	& = \sup_{s\geq 0}\left(j(s+1)-\mu\omega(e^s)\right) \leq \sup_{s\geq 0}\left(js-\mu\omega(e^{s-1})\right)\\ & \\ & \leq \lambda L + \sup_{s\geq 0}\left(js-\lambda\omega(e^{s})\right) = \lambda L + \lambda\varphi_\omega^\ast(\frac{j}{\lambda}).
\end{array}$$ Now proceed by induction on $N.$	
\end{proof}

It easily follows that, for every $f\in {\mathcal S}_{(\omega)}({\mathbb R})$ and $m\in {\mathbb N},$
$$
\sup_{j,k\in {\mathbb N}_0}\sup_{x\in {\mathbb R}}(1 + |x|)^k|f^{(j)}(x)|e^{-\lambda\varphi^\ast(\frac{j+k}{\lambda})} < +\infty.$$

Since ${\mathcal E}_{(\omega)}(\{0\})$ is a quotient of ${\mathcal S}_{(\omega)}({\mathbb R}),$ the next Lemma follows from \cite[Theorem 3]{nuclear} under the extra hypothesis that $\omega$ satisfies condition (\ref{eq:fromweighttosequence}). See also \cite{debrouwere}.

\begin{lemma}{\rm ${\mathcal E}_{(\omega)}(\{0\})$ is a Fréchet nuclear space. In particular, every bounded set in ${\mathcal E}_{(\omega)}(\{0\})$ is relatively compact.
}
\end{lemma}
\begin{proof}
We first observe that $\displaystyle{\mathcal E}_{(\omega)}(\{0\}) = \cap_{n=1}^\infty \ell^\infty (v_n)$ where $$v_n(j):= \exp\left(-n\varphi_\omega^\ast(\frac{j}{n})\right).$$ As in Lemma \ref{lem:cambioseminorma}, for every $m\in {\mathbb N}$ we take $\ell = Lm.$ Then
$$
\sum_{j=1}^\infty \frac{v_m(j)}{v_\ell(j)} \leq e^{mL}\sum_{j=1}^\infty e^{-j} < +\infty.$$ Now the conclusion follows from the Grothendieck-Pietsch criterion.
\end{proof}

The following result and its corollaries contrast with the behavior of the composition operators in the Schwartz class, where it is well known that if $f\in {\mathcal S}({\mathbb R})$ and $\psi$ is a non constant polynomial then $f\circ \psi\in {\mathcal S}({\mathbb R})$ (see for instance \cite[Theorem 2.3]{gj}).

\begin{theorem}\label{th:negativ}{\rm Let $\psi\in C^\infty({\mathbb R})$ satisfy $\displaystyle\lim_{x\to +\infty} \psi(x) = +\infty$ and $\psi'(x)\geq (\psi(x))^k$ for some $k > 0$ and every $x$ large enough. Let $\omega$ and $\sigma$ be two strong weights such that $\omega(t^{\frac{1}{k+1}}) = o(\sigma(t))$ as $t\to \infty.$ Then there exists $f\in {\mathcal S}_{(\omega)}({\mathbb R})$ such that $f\circ \psi\notin {\mathcal S}_{(\sigma)}({\mathbb R}).$
}
\end{theorem}
\begin{proof}
Proceeding by contradiction we assume that $f\circ \psi\in {\mathcal S}_{(\sigma)}({\mathbb R})$ for every $f\in {\mathcal S}_{(\omega)}({\mathbb R}).$ Then, by the closed graph theorem,
$$
C_\psi: {\mathcal S}_{(\omega)}({\mathbb R})\to {\mathcal S}_{(\sigma)}({\mathbb R}),\ f\mapsto f\circ \psi,$$ is a continuous operator. We put $\tilde{\omega}(t) = \omega(t^{\frac{1}{k+1}})$ and take $L\in {\mathbb N}$ so that $\tilde{\omega}(et)\leq L(1 + \tilde{\omega}(t),\ t\geq 0.$ Since $\tilde{\omega} = o(\sigma)$ then for every $n\in {\mathbb N}$ there is $s_n\in {\mathbb N}$ such that
$$
\varphi_\sigma^\ast(s) \leq 2Ln\varphi_{\tilde{\omega}}^\ast(\frac{s}{2Ln})\ \ \forall s\geq s_n.$$ We select $$j_n:=\max\left(2^n, s_{n+1}\right), n\in {\mathbb N}_0, \ \ I_m:=\{j\in {\mathbb N}:\ j_{m-1} < j \leq j_m\}, m\in {\mathbb N}.$$ Now, for every $j\in I_m,\ m\in {\mathbb N},$ we put
$$
\lambda_j:= m,\ \  a_j:= \exp(\lambda_j\varphi_\omega^\ast(\frac{j}{\lambda_j})).$$ The set
$$
{\mathcal C}:=\left\{(a_j\delta_{j\ell})_{\ell\in {\mathbb N}_ 0}:\ j > j_0\right\}$$ is bounded, hence relatively compact, in ${\mathcal E}_{(\omega)}(\{0\}).$ Since the Borel map $B:{\mathcal S}_{(\omega)}({\mathbb R})\to {\mathcal E}_{(\omega)}(\{0\})$ is surjective we can apply \cite[26.22]{mv} to find a relatively compact set ${\mathcal K} = \left\{g_j:\ j > j_0\right\}\subset {\mathcal S}_{(\omega)}({\mathbb R})$ such that ${\mathcal C} = B({\mathcal K}).$ Then ${\mathcal K}$ is a bounded set and $g_j^{(\ell)}(0) = a_j$ for $\ell = j$ while $g_j^{(\ell)}(0) = 0$ for $\ell\neq j.$ Now, for every $j > j_0$ we select $x_j > 0$ such that
$$
kj\log x_j - \lambda_j\omega(x_j) = \lambda_j\varphi_\omega^\ast(\frac{kj}{\lambda_j}).$$ This selection is possible because the continuity of $\omega$ and \ref{def:weight}($\gamma$) ensure that the supremum that appears in the definition of $\varphi_\omega^\ast$ is in fact a maximum. From the continuity of $C_\psi$ and Lemma \ref{lem:equicontinuous} we conclude that
$$
\left\{\exp\left(-\lambda_j\omega(x_j)\right)C_\psi(T_{x_j}g_j):\ j > j_0\right\}$$ is a bounded set in ${\mathcal S}_{(\sigma)}({\mathbb R}).$ In particular, for $h_j:= C_\psi(T_{x_j}g_j),$ we have
\begin{equation}\label{eq:cotapreliminar}
\sup_{j > j_0}\sup_{y\in {\mathbb R}}\exp\left(-\lambda_j\omega(x_j)\right)\left|h_j^{(j)}(y)\right|e^{-\varphi_\sigma^\ast(j)} < +\infty.\end{equation}
Since  $\displaystyle\lim_{s\to +\infty}\frac{\varphi_\omega^\ast(s)}{s} = +\infty$ then $\displaystyle\lim_{j\to\infty}x_j = +\infty$ and (for $j$ large enough, let say $j\geq J$) there is $y_j\in {\mathbb R}$ with $\psi(y_j) = x_j.$ From (\ref{eq:cotapreliminar}) we obtain
\begin{equation}\label{eq:cota}
\sup_{j > J}\exp\left(-\lambda_j\omega(x_j)\right)\left|h_j^{(j)}(y_j)\right|e^{-\varphi_\sigma^\ast(j)} < +\infty.\end{equation} To finish we will prove that (\ref{eq:cota}) does not hold. Indeed, by hypothesis and Lemma \ref{lem:faa} we get for all $j$ that
$$
h_j^{(j)}(y_j) = g_j^{(j)}(0)\left(\psi'(y_j)\right)^j \geq a_j x_j^{kj} = \exp\left(\lambda_j\varphi_\omega^\ast(\frac{j}{\lambda_j}) + kj\log x_j\right).$$ Since $\varphi_\omega^\ast$ is a convex function and $\varphi_\omega^\ast((k+1)s) = \varphi^\ast_{\tilde{\omega}}(s)$ we get
$$
\begin{array}{*2{>{\displaystyle}l}}
\exp(-\lambda_j\omega(x_j))	h_j^{(j)}(y_j) & \geq \exp\left(\lambda_j\varphi_\omega^\ast(\frac{j}{\lambda_j}) + \lambda_j\varphi_\omega^\ast(\frac{kj}{\lambda_j})\right) \\ & \\ & \geq \exp\left(2\lambda_j\varphi_\omega^\ast(\frac{(k+1)j}{2\lambda_j})\right) = \exp\left(2\lambda_j\varphi_{\tilde{\omega}}^\ast(\frac{j}{2\lambda_j})\right).\end{array}$$ From Lemma \ref{lem:cambioseminorma} we obtain
$$
2\lambda_j\varphi_{\tilde{\omega}}^\ast(\frac{j}{2\lambda_j}) \geq 2L\lambda_j\varphi_{\tilde{\omega}}^\ast(\frac{j}{2L\lambda_j}) + j - 2\lambda_j L \geq \varphi_\sigma^\ast(j) + j -2\lambda_j L.$$ The last inequality follows from the fact that for $j\in I_m$ we have $\lambda_j = m$ and $j > s_m.$ Consequently
$$
\exp\left(-\lambda_j\omega(x_j)\right)\left|h_j^{(j)}(y_j)\right|e^{-\varphi_\sigma^\ast(j)} \geq \exp(j - 2\lambda_j L),$$ which goes to infinity as $j\to \infty.$ This is due to the fact that if $j\in I_m$ then $\log_2(j)>m-1$ and so $\lambda_j<\log_2(j)+1$ for all $j$. We have reached a contradiction with (\ref{eq:cota}) and the proof is complete.
\end{proof}

\begin{corollary}\label{cor:pol-growth}{\rm Let $\psi\in C^\infty({\mathbb R})$ satisfy $\displaystyle\lim_{x\to +\infty} \psi(x) = +\infty$ and $\psi'(x)\geq (\psi(x))^k$ for some $k > 0$ and every $x$ large enough. Let $\omega$ be a strong weight such that condition (\ref{eq:fromweighttosequence}) is satisfied. Then there exists $f\in {\mathcal S}_{(\omega)}({\mathbb R})$ such that $f\circ \psi\notin {\mathcal S}_{(\omega)}({\mathbb R}).$
	}
\end{corollary}
\begin{proof}
	Take $s = k+1 > 1.$ It suffices to show that $\omega(t^{\frac{1}{s}}) = o(\omega(t))$ as $t\to \infty.$ In fact, for every $j\in {\mathbb N}$ we have
	$$
	2^j\omega(t) \leq \omega(H^j t) + H \sum_{k=0}^{j-1}2^k,\ \ t\geq 0.$$ For every $t\geq 1$ there is $j\in {\mathbb N}_0$ with $H^j \leq t^{s-1} < H^{j+1}.$ Hence
	$$
	\omega (t^s) \geq \omega(t H^j) \geq 2^j\omega(t)-(2^j-1)H,$$ from where it follows
	$$
	\frac{\omega(t)}{\omega(t^s)} \leq 2^{-j} + \frac{H}{\omega(t^s)}.$$ The conclusion easily follows.
\end{proof}

\begin{remark}\label{remark}{\rm The conclusion of Theorem \ref{th:negativ} is still valid if $\psi$ satisfies \par\noindent $\displaystyle\lim_{x\to +\infty}|\psi(x)| = +\infty$ and $|\psi'(x)|\geq c|\psi(x)|^k$ for some $c > 0$ and $x > 0$ large enough.
}
\end{remark}
In fact, let us assume for instance that $\displaystyle\lim_{x\to +\infty}\psi(x) = -\infty.$ Then $\psi'(x)$ is negative for $x$ large enough. Take $a:=c^{-1}$ and consider $\psi_a(x):=-\psi(ax).$ Then $\displaystyle\lim_{x\to +\infty}\psi_a(x) = +\infty$ and, for $x$ large enough, $\psi_a'(x) = a|\psi'(ax)| \geq (\psi_a(x))^k.$ Then there is $f\in {\mathcal S}_{(\omega)}({\mathbb R})$ such that $f\circ \psi_a\notin {\mathcal S}_{(\omega)}({\mathbb R}).$ Since ${\mathcal S}_{(\omega)}({\mathbb R})$ is invariant under dilations and reflections, the conclusion follows. $\Box$
\par\medskip
We recall that $\Sigma_d({\mathbb R}) \subset \Sigma_{d'}({\mathbb R})$ whenever $1 < d < d'.$

\begin{corollary}\label{cor:gevrey1}{\rm Let $\psi$ be given as in Remark \ref{remark}. Then for every $d\leq d' < (k+1)d$ there exists $f\in \Sigma_d({\mathbb R})$ such that $f\circ\psi\notin \Sigma_{d'}({\mathbb R}).$
}
\end{corollary}

The following result quantifies the loss of regularity when we compose a function in $\Sigma_d({\mathbb R})$ with a polynomial of degree greater than one.
\begin{corollary}\label{cor:_gevrey2}{\rm Let $\psi$ be a polynomial of degree  $N>1.$ Then for every $d\leq d' < \frac{2N-1}{N}d$ there is $f\in \Sigma_d({\mathbb R})$ such that $f\circ\psi\notin \Sigma_{d'}({\mathbb R}).$ In particular, for any polynomial $\psi$  of degree greater than one and $d\leq d' < \frac{3}{2}d$ there is $f\in \Sigma_d({\mathbb R})$ such that $f\circ\psi\notin \Sigma_{d'}({\mathbb R}).$
}
\end{corollary}
In fact, we can apply Corollary \ref{cor:gevrey1} with $k = \frac{N-1}{N} .$ Observe that $k+1 \geq \frac{3}{2}.$
\par\medskip
\begin{remark}{\rm The \emph{Gelfand-Shilov space} ${\mathcal S}_d(\mathbb R)$ ($d>1$)
of Roumieu type is the set of all smooth functions $f$ such that $$\sup \{ \frac {|x^\ell
f^{(j)}(x)|}{h^{\ell  + j}(\ell !\, j !)^d}: x\in {\mathbb R}, \ell, j\in {\mathbb N}_0\}$$ is finite for some $h>0.$ Endowed with the natural inductive topology, it is a DFS space. It is immediate to see that ${\mathcal S}_d({\mathbb R})\subset \Sigma_{d'}(\mathbb R)
$ with continuous inclusion whenever $d<d'.$ Hence, as a direct application of Corollary \ref{cor:_gevrey2}, the composition with polynomials of degree greater than one does not leave invariant the Gelfand-Shilov spaces of Roumieu type. In fact, if $\psi$ is a polynomial of degree  $N>1,$ and $d\leq d' < \frac{2N-1}{N}d$ there is $f\in {\mathcal S}_d({\mathbb R})$ such that $f\circ\psi\notin {\mathcal S}_{d'}({\mathbb R}).$
\par\medskip 
The behavior of composition operators on Gelfand-Shilov spaces of Roumieu type will be analyzed in a subsequent paper.}
\end{remark}

For weights of Gevrey type Corollary \ref{cor:pol-growth} can be improved.

\begin{theorem}\label{th:boundedderivative}{\rm Let $d > 1$ and $\psi\in C^\infty({\mathbb R})$ be given such that $C_\psi\left(\Sigma_d\right)\subset \Sigma_d.$ Then $\psi'$ is bounded.
	}
\end{theorem}
\begin{proof}
	We recall that $\Sigma_d = {\mathcal S}_{(\omega)}({\mathbb R})$ for $\omega(t) = t^{\frac{1}{d}}.$ An easy calculation gives
	$$
	\varphi_\omega^\ast(x) = xd\log\left(\frac{xd}{e}\right).$$ For every $m\in {\mathbb N}$ and $j\in {\mathbb N}_0$ we put $x_{m,j} = \left(\frac{jd}{m}\right)^d.$ 
	Proceeding by contradiction, let us assume that $\psi'$ is unbounded. Then we can find a sequence $(y_m)_m$ such that $\displaystyle\lim_{m\to \infty}|y_m| = +\infty$ and $|\psi'(y_m)| \geq 2^{md}.$ Without loss of generality we can assume $\psi(y_m) > 0.$ For every $m\in {\mathbb N}$ we put $x_m:= \psi(y_m)$ and select $j(m)\in {\mathbb N}$ such that
	$$
	x_{m,j(m)} \leq x_m < x_{m, j(m)+1},$$ or equivalently $j(m) \leq \frac{m}{d}\psi(y_m)^{\frac{1}{d}} < j(m)+1,$ from where it follows $\displaystyle\lim_{m\to \infty}j(m) = +\infty.$ Proceeding as in Theorem \ref{th:negativ} we can find a bounded sequence $(g_m)_m$ in $\Sigma_d$ such that $$g_m^{(j(m))}(0) = \exp\left(m\varphi_\omega^\ast(\frac{j(m)}{m})\right)$$ while $g_m^{(\ell)}(0) = 0$ for any $\ell\neq j(m).$ From the continuity of $C_\psi$ and Lemma \ref{lem:equicontinuous} we conclude that
	$$
	\left\{\exp\left(-m\omega(x_m)\right)C_\psi(T_{x_m}g_m):\ m\in {\mathbb N}\right\}$$ is a bounded set in $\Sigma_d.$ In particular, for $h_m:= C_\psi(T_{x_m}g_m),$ we have
	\begin{equation}
		\sup_{m\in {\mathbb N}}\exp\left(-m\omega(x_m)\right)\left|h_m^{(j(m))}(y_m)\right|e^{-\varphi_\omega^\ast(j(m))} < +\infty.\end{equation} By Lemma \ref{lem:faa} we get $$
	\sup_{m\in {\mathbb N}}2^{mj(m)d}\exp(m\varphi_\omega^\ast(\frac{j(m)}{m})-m\omega(x_{m,j(m)+1})-\varphi_\omega^\ast(j(m))) < +\infty.$$ That is,
	$$
	\sup_{m\in {\mathbb N}}\left(\frac{2^m}{me}\right)^{j(m)d} < +\infty,$$ which clearly is a contradiction.
\end{proof}

Our next result is an immediate consequence of the previous one plus Liouville's theorem and can be considered as a version of Beurling-Helson theorem (see for instance \cite{okoudjou}).

\begin{corollary}\label{cor:gevrey3}{\rm Let $\psi$ be an entire function with $\psi({\mathbb R})\subset {\mathbb R}.$ If $C_\psi(\Sigma_d )\subset \Sigma_d ,$ then $\psi$ is a polynomial of degree 1.}

\end{corollary}

The next result  shows that condition (a) in Proposition \ref{prop:sufficient} is reasonable. Its proof follows the steps of \cite[Theoren 2.3]{gj} and is due to \cite{vicent}.

\begin{proposition}\label{prop:necessary}{\rm Given  two weights $\sigma, \, \omega$ and $\psi \in C^\infty({\mathbb R})$ such that $C_\psi({\mathcal S}_{(\omega)}({\mathbb R}))\subset {\mathcal S}_{(\sigma)}({\mathbb R}),$ there exists $C$ such that $\sigma(x)\leq C(1+\omega(\psi(x)))$ for every $x\in {\mathbb R}.$}
\end{proposition}

\begin{proof} It suffices to proof that $\sigma(x)\leq C\omega(\psi(x))$ when $x$ is big enough.
Proceeding by contradiction we suppose that for every $n \in {\mathbb N}$ there exists $x_n$ such that $\sigma(x_n)>n\omega( \psi(x_n)).$ Passing to a subsequence if necessary we may assume $|\psi(x_n)|+1<|\psi(x_{n+1})|.$ We take $\rho \in {\mathcal S}_{(\omega)}({\mathbb R})$ with support contained in $[-\frac{1}{2},\frac{1}{2}]$ and $\rho(0)=1$ and define $$f(x)=\sum_{n=1}^\infty \frac{\rho(x-\psi(x_n))}{\exp(n\omega((\psi(x_n)))}.$$
Clearly $f\in C^\infty({\mathbb R}).$ We see that $f\in {\mathcal S}_{(\omega)}({\mathbb R}).$ Since $f$ vanishes outside of the union of the intervals $[\psi(x_n)-\frac{1}{2},\psi(x_n)+\frac{1}{2}]$ we fix $m, \ell \in {\mathbb N}$ and take $x\in [\psi(x_n)-\frac{1}{2},\psi(x_n)+\frac{1}{2}]$ with $n>Km,$ $K$ as in \ref{def:weight}($\alpha$). Then $$\begin{array}{ccl}
|f^{(j)}(x)|e^{m\omega(x)-\ell \varphi_\omega^\ast(\frac{j}{\ell})}& \leq & \displaystyle \frac{|\rho^{(j)}(x-\psi(x_n))|e^{Km\omega(x-\psi(x_n))-\ell \varphi_\omega^\ast(\frac{j}{\ell})}}{e^{n\omega(\psi(x_n))}} e^{Km(1+\omega(\psi(x_n)))}\\ \\
& \leq & \pi_{\ell, Km}(\rho) e^{Km} e^{(Km-n)\omega(\psi(x_n))}\\ \\
& \leq & \pi_{\ell, Km}(\rho) e^{Km}.\end{array}$$

On the other hand, $$(f\circ\psi)(x_n)e^{m\sigma(x_n)} = e^{m\sigma(x_n)-n\omega(\psi(x_n))} > e^{(m-1)\sigma(x_n)}$$ which is unbounded since $\lim_n |x_n|=\infty.$
\end{proof}
\noindent
\begin{corollary}{\rm Let $\psi \in C^\infty({\mathbb R})$ be given and assume that $C_\psi({\mathcal S}_{(\omega)}({\mathbb R}))\subset {\mathcal S}_{(\omega)}({\mathbb R})$.  \begin{itemize}
\item[(a)]  If $\omega$ satisfies condition (\ref{eq:fromweighttosequence}) then $|x|\leq C (1+|\psi(x)|)$ for some
$C>0.$
\item[(b)] If  $\omega(t)=\max(0, \log^s x)$, $s>1,$ then there exist $C, a >0$ such that $|x|\leq C(1+|\psi(x)|)^a.$\end{itemize}}
    \end{corollary}
    \begin{proof}
    Only (a) needs a proof. We already know that $\displaystyle\lim_{|x|\to \infty}|\psi(x)| = +\infty.$ Moreover, there is $ K$ such that  $\omega(x) \leq K \omega( \psi(
x))$ for $|x|>K.$ Then, by (\ref{eq:fromweighttosequence}), there is $C>0$ with $ K\omega(\psi(x))\leq \omega(C\psi(x))$
for $|x|>C$ which implies $|x|\leq C|\psi(x)|$ for $|x|>C.$

    \end{proof}

\section{Some sufficient conditions}
We present some conditions on $\psi\in C^\infty({\mathbb R}),$ $\omega$ and $\sigma,$ that are sufficient to guarantee that
$$
C_\psi:{\mathcal S}_{(\omega)}({\mathbb R})\to {\mathcal S}_{(\sigma)}({\mathbb R}), f\mapsto f\circ\psi,$$ is a well-defined operator, therefore continuous.
\par\medskip

\begin{proposition}\label{prop:sufficient}{\rm Let $\omega$ be a subadditive weight, $a\geq 1,$ and $\sigma(t) = \omega(t^\frac{1}{a}).$ We assume
\begin{itemize}
	\item[(a)] $|x|\leq C_0(1+|\psi(x)|)^a$ for every $x\in {\mathbb R}.$
	\item[(b)] For every $m\in {\mathbb N}$ there is $C_m > 0$ such that
	$$
	|\psi^{(j)}(x)| \leq C_m\exp(m\varphi_\sigma^\ast(\frac{j}{m}))(1 + \psi(x)|)^p\ \forall j\in {\mathbb N}\ \forall x\in {\mathbb R},$$ where $p = a-1.$
\end{itemize}
Then $C_\psi\left({\mathcal S}_{(\omega)}({\mathbb R})\right)\subset {\mathcal S}_{(\sigma)}({\mathbb R}).$
}
\end{proposition}
\begin{proof} Since the Gevrey weights are subadditive then it follows that also $\sigma$ is subadditive. Moreover $\varphi_\sigma^\ast(s) = \varphi_\omega^\ast(as),\ s\geq 0.$ Proceeding as in \cite[Page 403]{fg}, for each $m\in {\mathbb N}$ there is $B_m > 0$ such that
	$$
	\prod_{\ell=1}^j\left|\frac{\psi^{(\ell)}(x)}{\ell!}\right|^{k_\ell} \leq G_m(j,k,x):=B_m^k\left(1 + |\psi(x)|\right)^{kp}\frac{\exp(m\varphi_\sigma^\ast(\frac{j-k}{m}))}{(j-k)!},$$ where $\displaystyle \sum_{\ell=1}^j \ell k_\ell = j$ and $\displaystyle k = \sum_{\ell=1}^j k_\ell.$ \par As usual, we denote $\displaystyle I:=\left\{{\bm k} = (k_1, k_2, \ldots, k_j): \sum_{\ell=1}^j \ell k_\ell = j\right\}.$ We now fix $m\in {\mathbb N}$ and take $M\in {\mathbb N}$ and $D > 0$ satisfying
	$$
	C_0^qB_m^k\exp(M\varphi_\sigma^\ast(\frac{k+q}{M})) \leq D\exp(m\varphi_\sigma^\ast(\frac{k+q}{m}))$$ for every $k,q\in {\mathbb N}_0$ (Lemma \ref{lem:cambioseminorma}). Then
$$
\begin{array}{*2{>{\displaystyle}l}}
|x^q(f\circ\psi)^{(j)}(x)| & \leq C_0^q\left(1 + |\psi(x)|\right)^{aq}\sum_{{\bm k}\in I}\frac{j!}{k_1!\ldots k_j!}|f^{(k)}(\psi(x))|G_M(j,k,x)\\ & \\ & \leq C_0^q B_m^k	\sum_{{\bm k}\in I}\frac{j!}{k_1!\ldots k_j!}\exp(M\varphi_\omega^\ast(\frac{a(k+q)}{M}))\cdot \frac{\exp(m\varphi_\sigma^\ast(\frac{j-k}{m}))}{(j-k)!}\\ & \\ & \leq \sum_{{\bm k}\in I}\frac{j!}{k_1!\ldots k_j!}\frac{\exp(m\varphi_\sigma^\ast(\frac{j+q}{m}))}{(j-k)!}	\\ & \\ & \leq 4^j\exp(m\varphi_\sigma^\ast(\frac{j+q}{m})).
\end{array}$$ In the last inequality we used the fact that 
\begin{equation}\label{eq:sum}
\sum_{{\bm k}\in I}\frac{k!}{k_1!\ldots k_j!} = 2^{j-1},\end{equation} as follows, for instance, after taking $h(x) = \frac{x}{1-x}$ and evaluating $(h\circ h)^{(j)}(0)$ with Faà di Bruno's formula (\ref{eq:faa}). One more application of Lemma \ref{lem:cambioseminorma} gives
$$
\sup_{x\in {\mathbb R}}\sup_{q,j\in {\mathbb N}_0}|x^q(f\circ\psi)^{(j)}(x)|\exp(-m\varphi_\sigma^\ast(\frac{j+q}{m})) < +\infty$$ for every $m\in {\mathbb N}.$
\end{proof}

\begin{corollary}\label{sufficient_concave} {\rm Let $\omega$ be a subadditive weight and let us assume that $\psi\in C^\infty({\mathbb R})$ satisfies
		\begin{itemize}
			\item[(a)] $|x|\leq C_0(1+|\psi(x)|)$ for every $x\in {\mathbb R}.$
			\item[(b)] For every $m\in {\mathbb N}$ there is $C_m > 0$ such that
			$$
			|\psi^{(j)}(x)| \leq C_m\exp(m\varphi_\omega^\ast(\frac{j}{m}))\ \forall j\in {\mathbb N}\ \forall x\in {\mathbb R}.$$
		\end{itemize}
		Then $C_\psi\left({\mathcal S}_{(\omega)}({\mathbb R})\right)\subset {\mathcal S}_{(\omega)}({\mathbb R}).$
}
\end{corollary}

Condition (b) is related to the following.

\begin{definition}[\cite{bfg}]{\rm Let $\omega$ be a weight. Then $\mathcal{B}_{\infty,\omega}$ consists of those functions $f$ such that for every $m\in\mathbb{N}$ there exists $C_m$ satisfying \begin{equation*}
			|f^{(j)}(x)|\leq C_m\exp(m\varphi_\sigma^*(\frac{j}{m}))
		\end{equation*} for all $j\in\mathbb{N}_0$ and $x\in\mathbb{R}$.
}
\end{definition}

\begin{remark}{\rm If in Proposition \ref{prop:sufficient} we do not require that $\psi$ satisfies inequality (a) then we still conclude that $C_\psi\left({\mathcal S}_{(\omega)}({\mathbb R})\right)\subset \mathcal{B}_{\infty,\sigma}.$
}
\end{remark}

The following result means that the class ${\mathcal S}_{(\omega)}({\mathbb R})$ is stable under composition with functions that are appropriate perturbations of a multiple of the identity.

\begin{corollary}{\rm Let $\omega$ be a subadditive weight and let us assume that $\psi(x) = Ax + \lambda(x)$ for some $A\neq 0$ and $\lambda\in \mathcal{B}_{\infty,\omega}.$ Then $C_\psi\left({\mathcal S}_{(\omega)}({\mathbb R})\right)\subset {\mathcal S}_{(\omega)}({\mathbb R}).$
	}
\end{corollary}

Corollary \ref{cor:_gevrey2} means that Theorem \ref{te:2} is, in some sense, optimal.

\begin{theorem}\label{te:2}{\rm Let $\omega$ be a subadditive weight, $\sigma(t) = \omega(t^{\frac{1}{2}})$ and $\psi$ a non constant polynomial. Then $f\circ\psi\in {\mathcal S}_{(\sigma)}({\mathbb R})$ for every $f\in {\mathcal S}_{(\omega)}({\mathbb R}).$
	}
\end{theorem}
\begin{proof} Let $N$ be the degree of the polynomial $\psi.$ The result is trivial if $N=1,$ hence we assume that $N\geq 2.$ Then condition (b) in Proposition \ref{prop:sufficient} holds with $p=\frac{N-1}{N}$ hence it suffices to take $a=\frac{2N-1}{N}<2.$
\end{proof}

\begin{remark}{\rm Given a weight $\omega$, not necessarily subadditive, the inclusion  $C_\psi({\mathcal S}_{(\omega)}({\mathbb R}))\subset {\mathcal S}_{(\sigma)}({\mathbb R})$ holds for $\sigma(t) = \omega(t^{\frac{1}{3}})$ and $\psi$ a non constant polynomial. In fact, we can proceed as in Proposition \ref{prop:sufficient}, but instead of applying \cite[Page 403]{fg} and (\ref{eq:sum}) we use the identity $$
		\sum_{{\bm k}\in I}\frac{j!}{k_1!\ldots k_j!} = \sum_{k=1}^j{j-1\choose k-1}\frac{j!}{k!},$$ which follows from the main theorem in \cite{Lah}.
	}		
\end{remark}

For the weights $\omega(t)=\rm{max}(0,\log^s t)$ ($s>1$) we have the inclusion $C_\psi\left({\mathcal S}_{(\omega)}({\mathbb R})\right)\subset {\mathcal S}_{(\omega)}({\mathbb R})$ under mild assumptions.

\begin{corollary}\label{sufficient_log}{\rm Let $\omega(t):=\rm{max}(0,\log^s t)$ ($s>1$) and let us assume that $\psi \in C^\infty({\mathbb R})$ satisfies  for
\begin{itemize}
	\item[(a)] $|x|\leq C_0(1+|\psi(x)|)^{a_1}$ for every $x\in {\mathbb R}.$
	\item[(b)] For every $m\in {\mathbb N}$ there is $C_m > 0$ such that
	$$|\psi^{(j)}(x)| \leq C_m\exp(m\varphi_\omega^\ast(\frac{j}{m}))(1 + | \psi(x)|)^{a_2}\ \forall j\in {\mathbb N}\ \forall x\in {\mathbb R},$$
	\end{itemize} for some $a_1, \, a_2 >0. $ Then, $C_\psi\left({\mathcal S}_{(\omega)}({\mathbb R})\right)\subset {\mathcal S}_{(\omega)}({\mathbb R}).$
	}
\end{corollary}

\begin{proof} Take $a:=\max(a_1,a_2+1)$ and $\sigma(t) = \omega(t^\frac{1}{a}).$ Since $\varphi_\omega^\ast(s) = \varphi_\sigma^\ast(\frac{s}{a})$ then $\psi$ satisfies (a) and (b) in Proposition  \ref{prop:sufficient} and hence, since $\omega$ is equivalent to a subadditive weight, $C_\psi\left({\mathcal S}_{(\omega)}({\mathbb R})\right)\subset {\mathcal S}_{(\sigma)}({\mathbb R}).$ But ${\mathcal S}_{(\sigma)}({\mathbb R}) ={\mathcal S}_{(\omega)}({\mathbb R})$ since $\sigma(t)=\frac{1}{a^s}\omega (t).$
\end{proof}

\begin{corollary}\label{exponential}{\rm Let $\omega(t):=\rm{max}(0,\log^s t)$ ($s>1$) and  $\psi \in C^\infty({\mathbb R})$ such that for every $m\in {\mathbb N}$ there is $C_m > 0$ such that
	$$|\psi^{(j)}(x)| \leq C_m\exp(m\varphi_\omega^\ast(\frac{j}{m}))(1 + | \psi(x)|)\ \forall j\in {\mathbb N}\ \forall x\in {\mathbb R}.$$

If $\phi(x):=\exp(\psi(x)) \geq |x|^d$ for some $d>0,$ then $C_\phi\left({\mathcal S}_{(\omega)}({\mathbb R})\right)\subset {\mathcal S}_{(\omega)}({\mathbb R}).$}\end{corollary}
\begin{proof} We only have to verify condition (b) in the previous result. From the fact that $\psi(x)\geq 0$ for $|x|\geq 1$ we may find $C > 0$ such that $|\psi(x)|\leq C + \psi(x)$ for all $x\in {\mathbb R}.$ By Fa\`a di Bruno's formula, \cite[Lemma A.1 (viii)]{paley} and arguing as in the proof of Proposition \ref{prop:sufficient} we obtain for every $m \in {\mathbb N}$ a constant $B_m > 0$ such that for every $x\in {\mathbb R}$,   $$\begin{array}{*2{>{\displaystyle}l}}
		|\phi^{(j)}(x)|& \leq \phi(x) e^{m\varphi_\omega^\ast(\frac{j}{m})} B_m^j\sum_{{\bm k}\in I}\frac{j!}{k_1!\ldots k_j!} \frac{(C+\psi(x))^{k}}{(j-k)!} \\ & \\ & \leq \phi(x) B_m^j e^{m\varphi_\omega^\ast(\frac{j}{m})}\exp(C + \psi(x)) \sum_{{\bm k}\in I}\frac{j!}{k_1!\ldots k_j!}\frac{k!}{(j-k)!} \\ & \\ & \leq \phi^2(x)  e^{m\varphi_\omega^\ast(\frac{j}{m})}e^C j!B_m^j\sum_{{\bm k}\in I}\frac{j!}{k_1!\ldots k_j!}\frac{1}{(j-k)!}\\ & \\ & \leq \phi^2(x) e^{m\varphi_\omega^\ast(\frac{j}{m})}e^C (4B_m)^jj! \leq D_m \phi^2(x)  e^{m\varphi_\sigma^\ast(\frac{j}{m})} \end{array}$$ for some $D_m>0$ and for $\sigma(t)=\omega(t^{1/2}),$ where as usual $\displaystyle k = \sum_{\ell=1}^j k_\ell$ and $I:=\left\{{\bm k} = (k_1, k_2, \ldots, k_j): \sum_{\ell=1}^j \ell k_\ell = j\right\}.$ Now the conclusion follows from the fact that ${\mathcal S}_{(\sigma)}({\mathbb R}) ={\mathcal S}_{(\omega)}({\mathbb R}).$
\end{proof}

\begin{example}{\rm Taking $\psi(x)=x^2$ we get that $\phi(x)=e^{x^2}$ defines a continuous composition operator on ${\mathcal S}_{(\omega)}({\mathbb R})$ when $\omega(x)=\rm{max}(0,\log^s t)$ ($s>1$).}
\end{example}
To provide examples of functions $\psi$ satisfying the conditions of corollaries \ref{sufficient_concave} and \ref{sufficient_log} the next two results are useful.

\begin{corollary}\label{holomorphic_concave} {\rm Let $\psi$ be a holomorphic function on an open set containing the cone $C=\{z\in {\mathbb C}: |Im z|\leq L|Re z|\}.$ Assume that $\psi({\mathbb R})\subset {\mathbb R}$ and that there is $A>0$ such that

\begin{itemize}
\item[(a)] $|x|\leq A (1+|\psi(x)|)$ for every $x\in {\mathbb R},$
\item[(b)] $|\psi(z)|\leq A(1+|z|)$ for every $z\in C.$
\end{itemize}

Then, for every subadditive weight $\omega,$  $C_\psi\left({\mathcal S}_{(\omega)}({\mathbb R})\right)\subset {\mathcal S}_{(\omega)}({\mathbb R}).$}
\end{corollary}

\begin{proof} There is $\delta > 0$ such that, for every $x\in {\mathbb R}$ with $|x|\geq 1,$ the ball centered at $x$ with radius $r_x:= \delta |x|$ is contained in $C.$ By the Cauchy inequalities $$|\psi^{(j)}(x)|\leq j!\frac{1+|x|+r_x}{r_x^j},$$ for each $j\in {\mathbb N}.$ Then, there is a constant $C>0$ such that  whenever $|x|\geq 1$ one has $$|\psi^{(j)}(x)|\leq C j!\left(\frac{1}{\delta}\right)^{j-1}.$$ Since $\psi$ is real analytic in $[-1,1]$ we conclude the existence of $B>0$ such that  $$|\psi^{(j)}(x)|\leq j! B^{j+1} \mbox{ for each } x\in {\mathbb R}.$$ Therefore, using \cite[Lemma A.1 (viii)]{paley} we conclude that  $\psi$ fulfills (b) in Corollary \ref{sufficient_concave}.
\end{proof}

\begin{corollary}\label{holomorphic_log}{\rm Let $\omega(t):=\rm{max}(0,\log^s t)$ ($s>1$) and let us assume that $\psi \in C^\infty({\mathbb R})$ admits a holomorphic extension to the strip $H:=\{z\in {\mathbb C}: |Im z|<L\}$ for some $L>0$ and, for some $a_1, \, a_2 >0,$
\begin{itemize}
	\item[(a)] $|x|\leq C_0(1+|\psi(x)|)^{a_1}$ for every $x\in {\mathbb R}.$
	\item[(b)] For every $m\in {\mathbb N}$ there is $C_m > 0$ such that
	$$|\psi(z)| \leq C(1 + |z|)^{a_2}\  \forall z\in H.$$
	\end{itemize}  Then, $C_\psi\left({\mathcal S}_{(\omega)}({\mathbb R})\right)\subset {\mathcal S}_{(\omega)}({\mathbb R}).$
	}
\end{corollary}
\begin{proof} It is enough to use the Cauchy inequalities to check that Corollary \ref{sufficient_log} applies.
\end{proof}

\begin{example} {\rm \begin{itemize}
 \item[(a)] Let $\psi(x)=\sqrt{1+x^2}.$ Then $C_\psi\left({\mathcal S}_{(\omega)}({\mathbb R})\right)\subset {\mathcal S}_{(\omega)}({\mathbb R})$ for every subadditive weight $\omega.$

In fact, if $\log_0$ denotes the branch of the logarithm whose imaginary part takes values in $(-\pi,\pi),$ then $\psi(z):= \exp(\frac{1}{2}\log_0(1+z^2))$ is holomorphic in a neighborhood of $\{z\in {\mathbb C}: |Im z|\leq |Rez|\},$ and clearly satisfies conditions (a) and (b) in Corollary \ref{holomorphic_concave}.
\item[(b)] The functions $\psi(x)=(1+x^2)^a,$ $a>0$ or $\psi(x)=\frac{P(x)}{Q(x)}$, $P$ and $Q$ non-constant polynomials with $Q(x)\neq 0$ for every $x\in {\mathbb R}$ and $P$ having greater degree than $Q,$ satisfy the hypothesis in Corollary \ref{holomorphic_log}.
    \end{itemize}
}

\end{example}

\par\medskip\noindent
{\bf Acknowledgement.} The research was partially supported by the projects MCIN PID2020-119457GBI00/AEI/10.13039/501100011033 and GV Prometeu/2021/070.

\end{document}